\documentclass[12pt]{amsart}

\usepackage{amsmath}
\usepackage{amsfonts}
\usepackage{amsthm}

\usepackage[cmtip,arrow]{xy}
\usepackage{pb-diagram,pb-xy}

\newtheorem{theorem}{Theorem}

\newtheorem{lemma}{Lemma}
\newtheorem{proposition}{Proposition}
\newtheorem{definition}{Definition}

\newcommand{\C}{\mathcal{C}}

\newcommand{\F}{\mathcal{F}}

\renewcommand{\L}{\mathcal{L}}

\begin{document}
\title{Direct limits of Gorenstein injective modules}
%\author{***}

\author{Alina Iacob}
\address{A.I. \ Department of Mathematical Sciences \\
         Georgia Southern University \\
         Statesboro (GA) 30460-8093 \\ U.S.A.}
\email[Alina Iacob]{aiacob@GeorgiaSouthern.edu}

\begin{abstract}
One of the open problems in Gorenstein homological algebra is: when is the class of Gorenstein injective modules closed under arbitrary direct limits? It is known that if the class of Gorenstein injective modules, $\mathcal{GI}$, is closed under direct limits, then the ring is noetherian. The open problem is whether or not the converse holds. We give equivalent characterizations of $\mathcal{GI}$ being closed under direct limits. More precisely, we show that the following statements are equivalent:\\
(1) The class of Gorenstein injective left $R$-modules is closed under direct limits.\\
(2) The ring $R$ is left noetherian and the character module of every Gorenstein injective left $R$-module is Gorenstein flat.\\
(3) The class of Gorenstein injective modules is covering and it is closed under pure quotients.\\
(4) $\mathcal{GI}$ is closed under pure submodules.

\end{abstract}

\maketitle

\section{preliminaries}
Throughout the paper $R$ will denote an associative ring with identity. Unless otherwise stated, by \emph{module} we mean \emph{left} $R$-module.

We will denote by $Inj$ the class of all injective modules.
We recall that an $R$-module $M$ is Gorenstein injective if there exists an exact and $Hom(Inj, -)$ exact complex of injective modules\\ $\textbf{I}= \ldots \rightarrow I_1 \rightarrow I_0 \rightarrow I_{-1} \rightarrow \ldots $ such that $M = Ker (I_0 \rightarrow I_{-1})$.\\
We will use the notation $\mathcal{GI}$ for the class of Gorenstein injective modules.\\

Given a class of $R$-modules $\mathcal{F}$, we will denote as usual by $^\bot \F$ the class of all $R$-modules $N$ such that $Ext^1(N,F)=0$ for every $F \in \F$. The right orthogonal class of $\mathcal{F}$, $\mathcal{F}^\bot$, is defined in a similar manner: $\mathcal{F}^\bot = \{X, Ext^1(F,X)=0, \forall F \in \mathcal{F} \}$.\\

We recall the definitions for Gorenstein injective precovers, and covers. \\
\begin{definition} A homomorphism $\phi: G \rightarrow M$ is a Gorenstein injective precover of $M$ if $G$ is Gorenstein injective and if for any Gorenstein injective module $G'$ and any $\phi' \in Hom(G', M)$ there exists $u \in Hom(G', G)$ such that $\phi' = \phi u$. \\
A Gorenstein injective precover $\phi$ is said  to be a cover if any $u \in End_R(G)$ such that $\phi u = \phi$ is an automorphism of $G$.\\
\end{definition}

We recall that a pair $(\mathcal{L}, \mathcal{C})$ is a \emph{cotorsion pair} if $\mathcal{L} ^\bot = \mathcal{C}$ and $^\bot \mathcal{C} = \mathcal{L}$.\\
A cotorsion pair $(\mathcal{L}, \mathcal{C})$ is \emph{complete} if for every $_RM$ there exists exact sequences $ 0 \rightarrow C \rightarrow L \rightarrow M \rightarrow 0$ and $0 \rightarrow M \rightarrow C'\rightarrow L' \rightarrow 0$ with $C$, $C'$ in $\mathcal{C}$ and $L$, $L'$ in $\mathcal{L}$.\\

\begin{definition}
 A cotorsion pair $(\L, \C)$ is called hereditary
if one of the following equivalent statements hold:
\begin{enumerate}
\item $\L$ is resolving, that is, $\L$ is closed under taking kernels of epimorphisms.
\item $\C$ is coresolving, that is, $\C$ is closed under taking cokernels of monomorphisms.
\item $Ext^i (F, C) = 0$ for any $F \in \F$ and $C\in \C$ and $i\geq 1$.
\end{enumerate}
\end{definition}

It is known (\cite{saroch.stovicek}, Theorem 4.6) that $(^\bot \mathcal{GI}, \mathcal{GI})$ is a complete hereditary cotorsion pair over any ring $R$.

%de pus si GF si duality pairs.

Since we use Gorenstein flat modules as well, we recall their definition\\
\begin{definition}
A right $R$-module $G$ is Gorenstein flat if there is an exact complex of flat right $R$-modules $\textbf{F}= \ldots \rightarrow F_1 \rightarrow F_0 \rightarrow F_{-1} \rightarrow \ldots $ such that $\textbf{F} \otimes I$ is exact for every injective left $R$-module $I$, and such that $G = Ker(F_0 \rightarrow F_{-1})$.
\end{definition}

We will also use duality pairs. They were introduced by Holm and Jorgensen in \cite{holm:10:duality}. We recall the definition.\\
\begin{definition}(\cite{holm:10:duality})
A \emph{duality pair} over $R$ is a pair $(\mathcal{M},\mathcal{C})$, where $\mathcal{M}$ is a class of $R$-modules and $\mathcal{C}$ is a class of right $R$-modules, satisfying the following conditions:
\begin{enumerate}
\item $M \in \mathcal{M}$ if and only if $M^+ \in \mathcal{C}$ (where $M^+$ is the character module of $M$, $M^+ = Hom_Z (M, Q/Z)$).
\item $\mathcal{C}$ is closed under direct summands and finite direct sums.
\end{enumerate}
\end{definition}

A duality pair $(\mathcal{M},\mathcal{C})$ is called (co)product closed if the class $\mathcal{M}$ is
closed under (co)products in the category of all $R$-modules.

\begin{theorem}\cite{holm:10:duality}
Let $(\mathcal{M},\mathcal{C})$ be a duality pair. Then the following hold:
\begin{enumerate}
\item $\mathcal{M}$ is closed under pure submodules, pure quotients, and pure extensions.
\item If $(\mathcal{M},\mathcal{C})$ is coproduct-closed then M is covering.
\end{enumerate}
\end{theorem}

\section{results}

\begin{lemma}  If $K \in ^\bot \mathcal{GI}$, then $K^+ \in \mathcal{GF}^\bot$.
\end{lemma}

\begin{proof}

Let $K \in ^\bot \mathcal{GI}$. For any Gorenstein flat right $R$-module $F$, we have that $Ext^1(F,K^+) \simeq Ext^1(K, F^+) = 0$ (since, by \cite{holm:04:gordim}, Theorem 3.6, $F^+$ is a Gorenstein innjective module). So $K^+ \in \mathcal{GF}^\bot$.
\end{proof}

\begin{lemma} (this is basically \cite{iacob:twosided}, Theorem 4)
Let $R$ be a left noetherian ring such that the character module of every Gorenstein injective left $R$-module is a Gorenstein flat right $R$-module. Then $(\mathcal{GI}, \mathcal{GF})$ is a duality pair.
\end{lemma}

\begin{proof}
- We show first that $K $ is Gorenstein injective if and only if $K^+$ is Gorenstein flat.\\
One implication is a hypothesis that we made on the ring.\\
For the other implication: assume that $K^+$ is Gorenstein flat. Since $(^\bot \mathcal{GI}, \mathcal{GI})$ is a complete cotorsion pair (\cite{saroch.stovicek}, Theorem 4.6), there is an exact sequence $0 \rightarrow K \rightarrow G \rightarrow L \rightarrow 0$ with $G$ Gorenstein injective, and with $L \in ^\bot \mathcal{GI}$. This gives an exact sequence $0 \rightarrow L^+ \rightarrow G^+ \rightarrow K^+ \rightarrow 0$ with $K^+$ Gorenstein flat by hypothesis, and with $L^+ \in \mathcal{GF}^\bot$ (by Lemma 1). Thus $EXt^1(K^+, L^+) = 0$. Therefore $G^+ \simeq K^+ \oplus L^+$. Since $G^+$ is Gorenstein flat, it follows that $L^+ \in \mathcal{GF}$. So $L^+ \in \mathcal{GF} \bigcap \mathcal{GF}^\bot = Flat \bigcap Cotorsion$ (\cite{EIP}, Proposition 3.1). Thus $L^+$ is a flat module. Since the ring $R$ is noetherian, it follows that $L$ is injective.\\
The short exact sequence  $0 \rightarrow K \rightarrow G \rightarrow L \rightarrow 0$ with $G$ Gorenstein injective and with $L$ injective, gives that $G.i.d. K \le 1$. By \cite{christensen:06:gorenstein}, Lemma 2.18, there is a short exact sequence $0 \rightarrow B \rightarrow H \rightarrow K \rightarrow 0$ with $B$ Gorenstein injective, and with $i.d. H = G.i.d. K < \infty$. This gives a short exact sequence $0 \rightarrow K^+ \rightarrow H^+ \rightarrow B^+ \rightarrow 0$ with both $K^+$ and $B^+$ Gorenstein flat. Thus $H^+$ is also Gorenstein flat. Since $H$ has finite injective dimension, we have that $H^+$ has finite flat dimension. But then $H^+$ is Gorenstein flat of finite flat dimension.  By \cite{enochs:00:relative}, Corollary 10.3.4), $H^+$ is flat. Therefore $H$ is injective. \\
The short exact sequence $0 \rightarrow B \rightarrow H \rightarrow K \rightarrow 0$ with $B$ Gorenstein injective and with $H$ injective gives that $K$ is Gorenstein injective.\\
So $K $ is Gorenstein injective if and only if $K^+$ is Gorenstein flat.\\

- We can prove now that $(\mathcal{GI}, \mathcal{GF})$ is a duality pair: \\
By the above, $K $ is Gorenstein injective if and only if $K^+$ is Gorenstein flat. Since $\mathcal{GF}$ is closed under direct summands and direct sums, $(\mathcal{GI}, \mathcal{GF})$ is a duality pair.
\end{proof}

\begin{proposition}
Let $R$ be a left noetherian ring such that the character module of every Gorenstein injective left $R$-module is a Gorenstein flat right $R$-module. Then the class of Gorenstein injective left $R$-modules, $\mathcal{GI}$, is closed under direct limits.
\end{proposition}

\begin{proof}
By Lemma 2, $(\mathcal{GI}, \mathcal{GF})$ is a duality pair. By \cite{holm:10:duality}, Theorem 3.10 , $\mathcal{GI}$ is closed under pure submodules and pure quotients.\\
Since $\mathcal{GI}$ is closed under direct products, and since the direct sum of modules is a pure submodule of the direct product of the modules, it follows that $\mathcal{GI}$ is closed under arbitrary direct sums.\\
Since $\mathcal{GI}$ is closed under direct sums and under pure quotients, and since a direct limits of modules is a pure quotient of the direct sum of modules, it follows that $\mathcal{GI}$ is closed under direct limits.
\end{proof}

\begin{proposition}
If the class of Gorenstein injective left $R$-modules is closed under direct limits, then the ring $R$ is left noetherian and the character module of every Gorenstein injective left $R$-module is a Gorenstein flat right $R$-module.
\end{proposition}

\begin{proof}

%The class $\mathcal{GI}$ is closed under direct limits, so, in particular, it is closed under direct sums. By \cite{christensen:11:beyond}, , Proposition 3.15, the ring $R$ is left noetherian.\\
Since $(^\bot \mathcal{GI}, \mathcal{GI})$ is a complete hereditary cotorsion pair with $\mathcal{GI}$ closed under direct limits, it follows (by \cite{stovicek:11:telescoping}, Theorem 3.5) that $\mathcal{GI}$ is a definable class. Thus $\mathcal{GI}$ is closed under pure submodules. Since direct sums are pure submodules of direct products, it follows that $\mathcal{GI}$ is closed under direct sums. By \cite{christensen:11:beyond}, Proposition 3.15, the ring $R$ is left noetherian.\\\\
Also, since $\mathcal{GI}$ is definable, by \cite{EIP}, Lemma 1.1, we have that a module $M$ is Gorenstein injective if and only if the module $M^{++}$ is Gorenstein injective.

So consider a Gorenstein injective module $M$. By \cite{EIP}, Lemma 1.1, the module $M^{++} = (M^+)^+$ is Gorenstein injective. Since the ring $R$ is noetherian (therefore coherent), by \cite{holm:04:gordim}, Theorem 3.6, $M^+$ is a Gorenstein flat module.
\end{proof}

The following result is immediate from the fact that $\mathcal{GI}$ is the right half of a hereditary cotorsion pair, so it is closed under cokernels of monomorphisms.

\begin{lemma}
If the class of Gorenstein injective modules, $\mathcal{GI}$ is closed under pure submodules then it is also closed under pure quotients.
\end{lemma}

\begin{proof}
Let $0 \rightarrow M' \rightarrow M \rightarrow M" \rightarrow 0$ be a pure exact sequence with $M$ a Gorenstein injective module. By hypothesis, $M'$ is Gorenstein injective. Since the class $\mathcal{GI}$ is closed under cokernels of monomorphisms (\cite{saroch.stovicek}) it follows that $M"$ is also Gorenstein injective. So $\mathcal{GI}$ is also closed under pure quotients in this case.
\end{proof}

\begin{theorem}
The following statements are equivalent:\\
(1) The class of Gorenstein injective left $R$-modules, $\mathcal{GI}$, is closed under direct limits.\\
(2) The ring $R$ is left noetherian, and the character module of every Gorenstein injective left $R$-module is a Gorenstein flat right $R$-module.\\
(3) The class of Gorenstein injective left $R$-modules, $\mathcal{GI}$, is covering and closed under pure quotients.\\
(4) The class of Gorenstein injective left $R$-modules, $\mathcal{GI}$, is closed under pure submodules.\\
(5) $(\mathcal{GI}, \mathcal{GF})$ is a duality pair.
\end{theorem}

\begin{proof}
1 $\Rightarrow$ 2 is Proposition 2.\\
2 $\Rightarrow$ 1 is Proposition 1.\\
1 $ \Rightarrow$ 4. By \cite{stovicek:11:telescoping}, Theorem 3.5, $\mathcal{GI}$ being closed under direct limits implies that $\mathcal{GI}$ is a definable class, so it is closed under pure submodules (by definition).\\
4 $\Rightarrow$ 1. %Let $0 \rightarrow M' \rightarrow M \rightarrow M" \rightarrow 0$ be a pure exact sequence with $M$ a Gorenstein injective module. By hypothesis, $M'$ is Gorenstein injective. Since the class $\mathcal{GI}$ is closed under cokernels of monomorphisms (Saroch - St, Th ) it follows that $M"$ is also Gorenstein injective.
Since $\mathcal{GI}$ is closed under pure submodules, it follows that $\mathcal{GI}$ is also closed under pure quotients (by Lemma 3). Since $\mathcal{GI}$ is closed under direct products, and since the direct sum of modules is a pure submodule of the direct product of the modules, it follows that $\mathcal{GI}$ is closed under arbitrary direct sums.\\
Since $\mathcal{GI}$ is closed under direct sums and under pure quotients, and since a direct limits of modules is a pure quotient of the direct sum of modules, it follows that $\mathcal{GI}$ is closed under direct limits.\\
1 $\Rightarrow$ 3 By \cite{EEI}, Proposition 3, and \cite{christensen:06:gorenstein}, Proposition 3.15, if $\mathcal{GI}$ is closed under direct limits then $\mathcal{GI}$ is a covering class.\\
And by \cite{stovicek:11:telescoping}, Theorem 3.5, the class $\mathcal{GI}$ is definable, therefore it is closed under pure submodules. By Lemma 3, $\mathcal{GI}$ is also closed under pure quotients.
%Since $\mathcal{GI}$ is closed under cokernels of monomorphisms, it follows that it is also closed under pure quotients.

3 $\Rightarrow$ 1. Since $\mathcal{GI}$ is covering, it is closed under arbitrary direct sums. Since $\mathcal{GI}$ is closed under pure quotients and a direct limit of modules is a pure quotient of the direct sum of the modules, it follows that $\mathcal{GI}$ is closed under direct limits.\\

2 $\Rightarrow$ 5. By Lemma 2.\\

5 $\Rightarrow$ 4 . By \cite{holm:10:duality}, Theorem 3.10 , $\mathcal{GI}$ is closed under pure submodules.\\

\end{proof}

\textbf{Examples of noetherian rings such that the character module of any Gorenstein injective left $R$-module is a Gorenstein flat right $R$-module}.\\

- By \cite{holm:10:duality}, Lemma 2.5 (b), any commutative noetherian ring $R$ with
a dualizing complex has the desired property: the character modules of Gorenstein injective modules are Gorenstein flat.\\

- Any Gorenstein ring also has the desired property: if $R$ is an Iwanaga-Gorenstein ring, then $\mathcal{GI}$ is closed under direct limits (\cite{enochs:00:relative}, Lemma 11.1.2). Then the result follows form \cite{EIP}, Lemma 1.1.\\

We recall that a module $M$ is called strongly Gorenstein injective if there is an exact and $Hom(Inj, -)$ exact complex $\ldots \rightarrow E \rightarrow E \rightarrow E \rightarrow \ldots$ with $E$ injective, and with $M = Ker(E \rightarrow E)$.
\begin{theorem}
Let $R$ be a left noetherian ring such that $id.$ $R_R \le n$ for
some positive integer $n$. Then the character modules of Gorenstein injective left
$R$-modules are Gorenstein flat.
\end{theorem}

\begin{proof} Let $M$ be a strongly Gorenstein injective left $R$-module, and let $I$ be any
injective left $R$-module. By \cite{Iwanaga}, Proposition 1, the flat dimension of $I$ is less than or
equal to $n$. Then the right $R$-module $I^+$ has injective dimension $\le n$. It follows that
$Ext^i(M^+, I^+)=0$ for any $i \ge n+1$. So we have that $Ext^i(I, M^{++}) \simeq Ext^i(M^+, I^+)=0$
for all $i \ge n+1$, for any injective $_R I$.
But $M$ is strongly Gorenstein injective, so there exists an exact sequence $0 \rightarrow M \rightarrow E \rightarrow M \rightarrow 0$ with $E$ injective. This gives an exact sequence $0 \rightarrow M^{++} \rightarrow E^{++} \rightarrow M^{++} \rightarrow 0$, with $E^{++}$ an injective left R-module. It follows that $Ext^k (-, M^{++}) \simeq Ext^1(-, M^{++})$ for any $k \ge 1$. By the above, we have that $Ext^k (I, M^{++})=0$ for any
injective $_R I$, for any $k \ge 1$.
Then the exact sequence $ \ldots \rightarrow E^{++} \rightarrow E^{++} \rightarrow M^{++} \rightarrow 0$ is also $Hom(Inj, –)$ exact. So $M^{++}$ has an exact left injective resolution. Since $Ext^i (I, M^{++}) =0$ for all $i \ge 1$, for
any injective $_R I$, and $M^{++}$ has an exact left injective resolution it follows that $M^{++}$ is
Gorenstein injective ([8], Proposition 10.1.3).\\
Let $G$ be a Gorenstein injective left $R$ module. By [3], there exists a strongly
Gorenstein injective left $R$-module $M$ such that $M \simeq G \oplus H$. It follows that $G^{++}$ is
isomorphic to a direct summand of the Gorenstein injective module $M^{++}$, and therefore it is Gorenstein injective.
Since the ring $R$ is right coherent and $G^{++} = (G^+)^+$ is Gorenstein , it follows that $G^+$ is Gorenstein flat, for any Gorenstein injective $_R G$.

\end{proof}

- By \cite{iacob:gorflat}, Theorem 3, if $R$ is a two sided noetherian ring such that $_R R$ is a left $n$-perfect
ring, and there exists a dualizing module $_RV_R$ for the pair $(R, R)$, then the character
modules of Gorenstein injective left R-modules are Gorenstein flat right R-modules.\\

- By \cite{liu}, if $R$ is a left artinian ring such that the injective envelope of every
simple left $R$-module is finitely generated, then the character module of
Gorenstein injective left $R$-modules are Gorenstein flat.

\textbf{Some open questions:}\\

1. Does every (left ) noetherian ring satisfy the condition that character modules of Gorenstein injectives are Gorenstein flat? Or is this class of rings strictly contained in the class of noetherian rings?\\

2. Is the assumption that $\mathcal{GI}$ is covering implying that it is closed under pure quotients? In other words, is it true that the class of Gorenstein injective modules is closed under direct limits if and only if it is a covering class? This would be consistent with the Enochs' conjecture: "Every covering class is closed under direct limits".

%\bibitem{avramov:91:homological}

\end{document}